\documentclass[12pt]{article}

\usepackage{amsmath,amsthm,amssymb}

\def\real{\mathbb R}

\def\vecsp{{\mathbb V}}

\def\graph{{\mathcal G}}
\def\vertexset{{\mathcal V}_{\graph}}
\def\edgeset{{\mathcal E}_{\graph}}

\def\gbar{\overline \graph }

\def\jmp{\alpha (k)}

\def\alg{{\mathcal A}}

\def\dop{{\mathcal L}}

\def\wt{\widetilde}

\title{Metric Graphs \\
with Totally Disconnected Boundary}

\author{Robert Carlson
\\Department of Mathematics
\\University of Colorado at Colorado Springs
\\Colorado Springs, Colorado USA
\\rcarlson@uccs.edu}

\date{2020}

\begin{document}

%       Theorem environments

%% \theoremstyle{plain} %% This is the default
\newtheorem{thm}{Theorem}[section]
\newtheorem{cor}[thm]{Corollary}
\newtheorem{lem}[thm]{Lemma}
\newtheorem{prop}[thm]{Proposition}
\newtheorem{ax}{Axiom}

\theoremstyle{definition}
\newtheorem{defn}{Definition}[section]

\theoremstyle{remark}
\newtheorem{rem}{Remark}[section]
\newtheorem*{notation}{Notation}

%\numberwithin{equation}{section}

\newcommand{\thmref}[1]{Theorem~\ref{#1}}
\newcommand{\secref}[1]{\S\ref{#1}}
\newcommand{\lemref}[1]{Lemma~\ref{#1}}
\newcommand{\propref}[1]{Proposition~\ref{#1}}
\newcommand{\figref}[1]{Figure~\ref{#1}}

\numberwithin{equation}{section}

%\copyrightinfo{2000}{American Mathematical Society}

%******************

\maketitle

\begin{abstract}

Boundary analysis is developed for a rich class of generally infinite weighted graphs with compact metric completions. 
These graph completions have totally disconnected boundaries.
The classical notion of $\epsilon$-components and the existence of suitable measures are used to construct generalized Haar bases 
and Hilbert spaces of functions on the boundaries.  Suitable exit measures are constructed and analyzed using harmonic functions.  

\end{abstract}


MSC-class: 34B45 (primary, 05C63, 51F99 (secondary)

%\tableofcontents

\newpage

\section{Introduction}

Let $\graph $ denote a connected, locally finite metric graph with a countable vertex set $\vertexset$ and edge set $\edgeset$.
Our focus is primarily on infinite graphs $\graph $, although finite graphs are not excluded.  
Edges $e \in \edgeset $ are assigned a length $l_e$ and are identified with a real interval $[a_e,b_e]$ of length $l_e$. 
In this work $0 < l_e < \infty $.  The lengths of edges and their subsets are extended to path lengths by addition as usual.
The distance $d(x_1,x_2)$ between points $x_1$ and $x_2$ in $\graph $ is defined as     
the infimum of the lengths of paths joining $x_1$ and $x_2$.   

The boundary $\partial \graph$ of $\graph $ is defined to be the set of vertices with degree $1$.
(This definition will later be extended for explicitly finite graphs.)
Define the interior $\graph _{int}$ of ${\graph}$ to be the complement of the boundary vertices.
As a metric space, ${\graph}$ has a completion $\gbar $, which is assumed to be compact.   
The boundary of $\gbar $ will be $\partial \gbar  = \gbar  \setminus \graph _{int}$.

This work treats two types of questions.  The first aims to understand which compact metric graph completions $\gbar$
have totally disconnected boundary $\partial \gbar$.    
One result is that $\partial \gbar $ is totally disconnected if and only if $\gbar $ is weakly connected,
which means that any two distinct points of $\gbar $ can be separated by the removal of a finite set of points from $\graph $. 
Another result shows that $\partial \gbar $ is totally disconnected if and only if $\gbar $ is homeomorphic to its end compactification.
Totally disconnected compact metric spaces have a rich collection of
clopen sets, that is sets which are both open and closed.  This fact is an essential element of the subsequent analysis. 

The second type of question considers the foundations for function theory on the boundary $\partial \gbar$.
When $\partial \gbar$ is totally disconnected, several constructions from the theory of totally disconnected compact metric spaces
can be used to good effect.   The classical notion of $\epsilon$-components provides a canonical sequence of nested partitions of $\partial \gbar$, 
each partition consisting of finitely many clopen sets.  Given a measure $\mu $ on $\partial \gbar$ which is positive on nonempty clopen sets, 
one may construct a generalized Haar basis from functions which are constant on the sets of the partition, yielding  
an orthonormal basis for the Hilbert space $L^2(\mu)$.

The developments continue by exploiting work on the Dirichlet problem \cite{Carlson08,Carlson12}.
The Dirichlet to Neumann map is used to construct exit measures on $\partial \gbar$ which are suited to the earlier Hilbert space constructions.
Exit measures are shown to assign positive weight to nonempty clopen sets.
Compressed Dirichlet to Neumann maps based on finite clopen partitions of the boundary $\partial \gbar $ are the limits
of corresponding maps on finite subgraphs of $\graph $.

Metric graphs $\graph$ are closely related to reversible Markov chains and resistor networks, which have an extensive probability literature \cite{Doyle, LP,Woess},
Probabilistic considerations lead to a general boundary theory, which seems rather abstract.   
Here, harmonic functions and metric space methods are used to develop our main themes linking totally disconnected boundaries, 
Hilbert spaces of functions on the boundary, and the generalized Haar bases suitable for multiresolution analysis. 
Work with overlapping themes can be found in \cite{Kigami}, where there is an extensive analysis for tree networks using probabilistic methods.

The last section of this work will make use of results from \cite{Carlson08,Carlson12}.
Note that Proposition 4.8 in \cite{Carlson12} is not correct, but this does not effect the rest of the work.  
The author gratefully acknowledges an assist from Greg Morrow, 
who pointed out the relevance of Kakutani's Theorem \cite{Kakutani}.

\section{Metric graph boundary}

The metric graph completions $\gbar $ studied here will satisfy two main conditions.  
First, $\gbar $ is compact.  Since $\gbar$ must be totally bounded, $\gbar$ is compact if and only if for every $\epsilon > 0$ 
there is a subgraph $\graph _{0}$ of $\graph $, with finitely many edges and vertices,
such that for every $y \in \graph $ there is an $x \in \graph _{0}$ with $d(x,y) < \epsilon $.

The second condition is that $\gbar $ is weakly connected.  This less standard condition
means that for every pair of distinct points $x,y \in \gbar $,
there is a finite set of points $W = \{ w_1,\dots ,w_K \} $ in the graph $\graph $ 
separating $x$ from $y$.  That is, there are disjoint open sets $U_x,V_y$ with $x \in U$ and $y \in V$ such that
$\gbar \setminus W = U_x \cup V_y$.

Metric graphs $\graph $ with finite volume, that is $\sum_{e \in \edgeset} l_e < \infty$,
have compact and weakly connected completions \cite{Carlson08}. 
Completions of metric trees are weakly connected, so if compact (which is not always the case), they satisfy our conditions. 
The finite volume condition is not necessary; in particular it is easy to construct metric trees having infinite volume, but with compact completions.  

A key role in the analysis of $\gbar$ and $\partial \gbar$ is played by the rich collection of clopen sets in $\partial \gbar$, 
that is subsets of $\partial \gbar$ which are both open and closed.
The discussion begins with a result about partitions of a totally disconnected compact metric space $\Omega $.
The boundary $\partial \gbar$ is a closed set, so $\partial \gbar$ is compact if $\gbar$ is.

Suppose ${\cal E }_1$ and ${\cal E }_2$ are partitions of $\Omega$.
Partition ${\cal E }_2$ is a refinement of $ {\cal E }_1$ if each set in ${\cal E }_2$ is a subset of a set in $ {\cal E }_1$.
A version of the next result is in \cite[p. 97]{Hocking}.

\begin{prop} \label{manycor}
Suppose $\Omega $ is a totally disconnected compact metric space.  For any $\epsilon > 0$,
there is a finite partition ${\cal E} = \{ E(n), n=1, \dots ,N \}$ of $\Omega $ by clopen sets such that 
the diameter of each $E(n) \in {\cal E}$ is less than $\epsilon $.

There is a sequence $\{ {\cal E }_j, j = 1,2,3,\dots \} $ of partitions of $\Omega $ by clopen sets, with ${\cal E }_{j+1}$ a refinement of ${\cal E }_j$,
such that the diameter of each set in $ {\cal E }_j $ is less than $2^{-j}$. 
\end{prop}

The next result characterizes the compact completions $\gbar $ with totally disconnected boundary.
Introduce the notation $B_{\epsilon}(z)$ for the open ball of radius $\epsilon > 0$ centered at $z \in \gbar $, 
while the $\epsilon $ neighborhood of a set $E \subset \gbar$ is  
\[N_{\epsilon}(E) = \bigcup_{z \in E} B_{\epsilon}(z) .\]

\begin{thm} \label{tdeqwc}
If $ \gbar $ is compact,  then $\partial \gbar $ is totally disconnected if and only if 
$\gbar $ is weakly connected.
\end{thm}

\begin{proof}

From the definition of weakly connected graph completions, if $x$ and $y$ are distinct points in $\partial \gbar $,
they are in the distinct clopen sets $U_x \cap \partial \gbar $ and $V_y \cap \partial \gbar $. 
Since $U_x \cup V_y = \partial \gbar$, $x$ and $y$ lie in distinct connected components.
The boundary $\partial \gbar$ is thus totally disconnected if $\gbar $ is weakly connected.  

Assume now that $\gbar $ is compact with a totally disconnected boundary.  
Consider distinct points $x,y$ in $\gbar $. 
If either $x$ or $y$ is a point of $\graph $, they are easily separated by the removal of a finite set of points in $\graph $.
If both $x$ and $y$ belong to $\partial \gbar$, then by \propref{manycor} there are disjoint clopen sets $E_x,E_y$ 
with $x \in E_x$, $y \in E_y$, and $E_x \cup E_y = \partial \gbar$.
Since $E_x$ and $E_y$ are compact and disjoint in $\gbar $, the neighborhoods $N_{\epsilon}(E_x)$ and $N_{\epsilon}(E_y)$
are disjoint if $\epsilon > 0$ is sufficiently small \cite[p. 86]{KolmF}.

Since $\gbar$ is compact, the set $\gbar \setminus [ N_{\epsilon}(E_x) \cup N_{\epsilon}(E_y)] $
is contained in a set $S$ which is the union of finitely many closed edges of $\graph $.  
Define 
\[U_x =  N_{\epsilon}(E_x) \setminus S, \quad U_y =  N_{\epsilon}(E_y) \setminus S.\]
Since $x,y \in \partial \gbar$, these sets are still open neighborhoods of $x,y$ respectively. 

Let $W$ be the set of vertices in $S$, and let $V$ be the complement of $U_x$ in $\gbar \setminus W$.
$V$ is open since it is the union of $U_y$ and the collection of open edges of $S$. 
The sets $U_x,V$ provide the desired separation of $x$ and $y$ by a finite set $W$ of points from $\graph $, showing that $\gbar $ is weakly connected.
\end{proof}

\begin{thm} \label{homeo}
Assume that the metric graphs $\graph _1, \graph _2$ have the same vertex and edge sets, but possibly different edge lengths.
Suppose the metric completions $\gbar _1,\gbar _2$ are both compact, while $\graph _1$ has finite volume.
Then $\gbar _1$ and $\gbar _2$ are homeomorphic if and only if $\gbar _2 $ is weakly connected.
\end{thm}

\begin{proof}
First, suppose $\gbar _1$ and $\gbar _2$ are homeomorphic.  Since $\gbar _1$ is weakly connected \cite[Thm 2.9]{Carlson08},
any distinct points $x,y$ in $\gbar _1$ may be separated by the removal of a finite set of points in $\graph _1$.
This separation is preserved under homeomorphism, so $\gbar _2$ is weakly connected.

In the other direction, assume $\gbar _2$ is weakly connected. 
Without loss of generality, assume that for each edge, the length in $\graph _1$ is smaller than the length of the corresponding edge in $\graph _2$.

Suppose a (directed) edge $e \in \graph _1$ is identified with a copy of $[0,a_e]$, while the corresponding edge in $\graph _2$
is identified with a copy of $[0,b_e]$.  The map $\phi : \graph _2 \to \graph _1$ which is the identity on vertices, but takes
$x \in [0,b_e]$ to $a_ex/b_e \in [0,a_e]$ is a distance reducing continuous bijection.
Since $\phi $ is uniformly continuous, it extends continuously \cite[p. 149]{Royden} to a map $\phi :\gbar _2 \to \gbar _1$
which is also distance reducing.

If $x_1 \in \gbar _1 \setminus  \graph _1$,  then there is a sequence $s_1 = \{ v_k \}$ of vertices in $\graph _1$ with limit $x_1$.  
The sequence $s_1$ lifts to a sequence $s_2 = \{ \phi ^{-1} (v_k) \}$ in $\graph _2$.  
Since $\gbar _2$ is compact, $s_2$ has a subsequence $s_3$ with a limit $y \in \gbar _2$ .
Since $\phi $ is distance reducing, $\phi (y)$ must be the limit of the subsequence $\phi (s_3)$ of $s_1$.  
That is, $\phi (y) = x_1$, and $\phi :\gbar _2 \to \gbar _1$ is surjective. 

Suppose $y_1,y_2$ are distinct points in $\partial \gbar _2$.  
Since $\gbar _2$ is weakly connected, there is a finite set of points $W = \{x_1,\dots ,x_K \} \subset \graph _2$ separating $y_1$ and $y_2$.
Let $U,V$ be disjoint and open, with $U \cup V = \gbar _2 \setminus W$, and with $y_1 \in U$ and $y_2 \in V$.  
For any $\epsilon > 0$ there are points $\xi _1,\xi _2 \in \graph _2$ such that $\xi _1 \in U$, $\xi _2 \in V$,
$d_2(\xi _1,y_1) < \epsilon $ and  $d_2(\xi _2,y_2) < \epsilon $.
Let $\{ e_j, j = 1,\dots ,J\} \subset \graph _2$ be the set of edges containing at least one point $x_k \in W$.
If $\epsilon $ is sufficiently small, any path from $\xi _1$ to $\xi _2$ must include at least one edge $e_j$.

Since $\phi (e_j) \subset \graph _1$, there is a $\sigma > 0$ such that $d_1 (\phi (\xi _1),\phi (\xi _2)) > \sigma $.  Pick $\epsilon < \sigma /2$.
Since $\phi $ is distance reducing, $d_1(\phi (y_1),\phi (\xi _1) ) < \epsilon  $ and  $d_1(\phi (y_2),\phi (\xi _2) ) < \epsilon  $, so
\[d_1(\phi (y_1),\phi (y_2) ) > \sigma - 2 \epsilon > 0,\]
and $\phi : \gbar _2 \to \gbar _1$ is injective.

Since a continuous bijection of compact metric spaces is a homeomorphism \cite[p. 94]{KolmF} \cite[p. 103]{Hocking}, 
$\phi $ is  a homeomorphism.
 
\end{proof}

Georgakopoulos \cite[Thm 1.1]{Georg11} recognized that completions $\gbar $ with finite volume 
are homeomorphic to the end compactification \cite[p.226]{Diestel} of $\graph $, a compactification 
developed in \cite{Freud} and \cite{Halin}.  
The next corollary follows from \thmref{homeo} and this result.
 
\begin{cor}
$\gbar $ is homeomorphic to the end compactification of $\graph $ if and only if
$\gbar $ is compact and weakly connected.
\end{cor}

Again assume that $ \gbar $ is compact and weakly connected.  For $\epsilon > 0$ let $\graph _{\epsilon}$ denote the finite subgraph of $\graph $ 
whose (closed) edges are those containing a point whose distance from $\partial \gbar $ is at least $\epsilon $.
The vertices of $\graph _{\epsilon}$ are inherited from these edges.
Define the relative boundary of $\graph  _{\epsilon}$ in $\graph $ to be the set of vertices $v \in \graph _{\epsilon} $ 
such that either (i) $v$ is a leaf of $\graph $, or
(ii) $v$ has an incident edge $e_v \in \graph \setminus \graph _{\epsilon}$.

As noted in the proof of \thmref{tdeqwc}, 
if ${\cal E} = \{ E(n),n = 1,\dots ,N \}$ is any finite partition of $\partial \gbar $ by clopen sets,
and $\epsilon > 0$ is sufficiently small, then the $\epsilon $ neighborhoods $N_{\epsilon}(E(n))$ are pairwise disjoint.

\begin{prop} \label{partition}
Suppose ${\cal E} = \{ E(n), n = 1,\dots ,N \}$ is a finite partition of $\partial \gbar $ by clopen sets, while $\epsilon > 0$ is sufficiently small.  
Then every point $v$ in the relative boundary of $\graph _{\epsilon}$ is connected to 
$\partial \gbar$ by a path $\gamma $ containing no edge of $\graph _{\epsilon}$,
and every such path starting at $v$ lies in the same set $N_{\epsilon}(E(n))$
\end{prop}

\begin{proof}

In case (i), $v \in \partial \gbar $ and the only such path is the trivial path $\gamma (t) = v$.

Suppose (ii) holds.  Since $v$ has an incident edge which is not in  $\graph _{\epsilon}$, 
$v$ has distance less than $\epsilon $ from $\partial \gbar$.  That is, $v \in N_{\epsilon}(E_n)$ for some $n$,
and there is a path $\gamma $ from $v$ to $\partial \gbar $ in $N_{\epsilon}(E_n)$.
If $\gamma _1$ is another path from $v$ to $\partial \gbar $ containing no edge of $\graph _{\epsilon}$,
then every point of $\gamma _1$ has distance less than $\epsilon $ from $\partial \gbar $,
so stays in the set $N_{\epsilon}(E_n)$.      

\end{proof}

\section{Partitions and orthonormal bases }

Partitions and nested sequences of partitions play an important role in developing function theory for $\partial \gbar$.   
The developments in this section apply to compact, totally disconnected metric spaces $\Omega $, not just $\partial \gbar$.
For such spaces $\Omega $ there is a canonical procedure providing a nested sequence of partitions.

Using an idea that apparently goes back at least to G. Cantor \cite[p. 108]{Hocking}, define 
an $\epsilon $-chain from $y_1$ to $y_2$ to be a finite sequence of points $x_k$, $k=1,\dots ,K$ in $\Omega $ with
$x_1 = y_1$, $x_K = y_2$, and $d(x_k,x_{k+1}) < \epsilon $.   Being connected by an $\epsilon $-chain is an equivalence relation;  
call the equivalence classes $\epsilon $-components of $\Omega $.
For $\epsilon > 0$ let ${\cal C}_{\epsilon}$ denote the set of $\epsilon $-components of $\Omega $.
Note that if $\epsilon $ exceeds the diameter of $\Omega $, then ${\cal C}_{\epsilon}$ consists of the singleton $\Omega $.

For $E \subset \Omega $, let $E^c = \Omega \setminus E$ denote the complement of $E$.
If $E$ is an $\epsilon $-component of $\Omega $, then $d(x,y) \ge \epsilon $ for $x \in E$ and $y \in E^c$.
Thus both $E$ and $E^c$ are clopen.  The $\epsilon$-components provide a useful method \cite{Robins} for partitioning $\Omega $.

\begin{prop} \label{specialpart}
Suppose $\Omega $ is a nonempty, compact, totally disconnected metric space.
For each $\epsilon > 0$ the set ${\cal C}_{\epsilon}$ of $\epsilon $-components is a finite partition of $\Omega $ by clopen sets.
If $\epsilon < \sigma $, then ${\cal C}_{\epsilon }$ is a refinement of ${\cal C}_{\sigma }$.
The maximum diameter of a set in ${\cal C}_{\epsilon}$ has limit zero as $\epsilon \to 0^+$.
\end{prop}

\begin{proof}
As noted above, the $\epsilon $-components are clopen.  Since $\Omega $ is compact and
${\cal C}_{\epsilon}$ is a partition of $\Omega $ by open sets, the set ${\cal C}_{\epsilon}$ is finite.

If $\epsilon < \sigma $ and $E \in {\cal C}_{\epsilon}$, then $E$ is contained in the $\sigma $-component of its elements, 
so ${\cal C}_{\epsilon }$ is a refinement of ${\cal C}_{\sigma }$,
  
Finally, using \propref{manycor}, for any $\sigma > 0$ there is a finite partition ${\cal E} = \{ E(n), n=1,\dots, N \}$ of $\Omega $ by clopen sets such that 
the diameter of each $E(n) \in {\cal E}$ is less than $\sigma $.  Since $E(n),E(n)^c$ are compact, for $x \in E(n)$ and $y \in E(n)^c$,
$r_n = \min d(x,y) > 0 $.   Let $r = \min _n r_n$ and $\epsilon = \min (r,\sigma )$.
Each set in $E \in {\cal C}_{\epsilon }$ is contained in one $E(n)$, so ${\rm diam}(E) \le \sigma $.
\end{proof}

A sequence of partitions ${\bf E} = \{ {\cal E} _n, n = 0,1,2,\dots \}$ is nested if each partition ${\cal E}_{n+1}$ 
is a refinement of ${\cal E}_{n}$.  By using $\epsilon $-components one may identify a canonical sequence of partitions of $\Omega $ by clopen sets. 
Since the number $\eta (\epsilon )$ of $\epsilon $-components is increasing as $\epsilon \to 0^+$, there is a (possibly finite) decreasing sequence $\jmp$ 
listing the values of $\epsilon$ where the size of the partition increases, with each number appearing with suitable multiplicity.  

Define the separation of two disjoint compact sets $K_1,K_2$ in $\Omega $ to be
\[s(K_1,K_2) = min (x_1,x_2), \quad x_1 \in K_1, x_2 \in K_2.\]
Since the sets are compact and disjoint, the minimum is achieved and positive.

\begin{prop} If $\Omega $ is a nonempty, compact, totally disconnected metric space,
the number of $\epsilon $-components is continuous from the left,
\[\lim_{\epsilon \uparrow \epsilon _0} \eta (\epsilon ) = \eta (\epsilon _0 ).\]
\end{prop}

\begin{proof}

Suppose 
\[ \eta _1 =  \inf _{\epsilon < \epsilon _0} \eta (\epsilon ).\]
Recall that the $\epsilon $-components grow, in the sense of inclusion, as $\epsilon $ increases.
Thus there is some $\epsilon _1 < \epsilon _0$ such that the set  
$\{E(n), n = 1,\dots ,\eta _1 \}$ of $\epsilon $-components does not change if $\epsilon _1 < \epsilon < \epsilon _0 $.
This implies that for all $E(n)$, the separation of $E(n)$ from $E(n)^c$ is at least $\epsilon _0$.
But then the number of $\epsilon$-components does not decrease if $\epsilon = \epsilon _0$.
\end{proof}

As the canonical sequence, take the partitions ${\cal E}_0 = \Omega$, and for $n = 1,2,3,\dots $ let ${\cal E}_n = {\cal C}_{\jmp}$.

For a partition ${\cal E} = \{ E(m), m = 1, \dots , M \}$ of $\Omega $ by clopen sets, 
let ${\vecsp}_{\cal E}$ denote the vector space spanned by the characteristic functions of the sets $E(m) \in {\cal E}$.
 If $\bf E$ is a nested sequence of partitions, then ${\vecsp}_{{\cal E}_n} \subset  {\vecsp}_{{\cal E}_{n+1}}$;
let ${\vecsp}_{\bf E}$ denote the vector space ${\vecsp}_{\bf E} = \bigcup_{n=0}^{\infty} {\vecsp}_{{\cal E}_n}$.

If $\bf E$ is a nested sequence of finite partitions by clopen sets, 
let $B$ denote the collection of sets consisting of finite unions of sets in any of the partitions ${\cal E}_n$.  If $S \in B$ is a finite union of sets from the partitions, 
then $S$ can be written as the finite union of sets from a single sufficiently fine partition ${\cal E}_N$.
Since ${\cal E}_N$ is a finite partition by clopen sets, the complement of $S$ is also in $B$.
Since $B$ is closed under finite unions and complements, $B$ an algebra of sets \cite[p. 20]{Folland}.  
Borrowing a term from calculus, let the mesh of the partition $m_n$ be 
the maximum of the diameters of the sets $E(m,n) \in {\cal E}_n$.

\begin{thm} \label{Borel}
Assume $\Omega $ is a nonempty, compact, and totally disconnected metric space with a nested sequence  
$\bf E$ of finite partitions by clopen sets.
If $\lim _{n \to \infty} m_n = 0$, then 
(i) the uniform closure of ${\vecsp}_{\bf E}$ is the set of continuous functions on $\Omega $, and 
(ii) the $\sigma$-algebra $\cal B$ generated by $B$ is the Borel sets of $\Omega $.
\end{thm}

\begin{proof}
Since the sets in the partitions are clopen, the functions in ${\vecsp}_{\bf E}$ are continuous.
Moreover, ${\vecsp}_{\bf E}$ is an algebra with operations of pointwise addition and multiplication.
Since  $\lim _{n \to \infty} m_n = 0$, the algebra separates points and does not vanish anywhere.
Thus (i) holds by the Stone-Weierstrass Theorem. 

Each set in $B$ is open, so $\cal B$ is a subset of the Borel sets. 
Suppose $K$ is a closed set in $\Omega $.  For a positive integer $N$, chose a partition ${\cal E}_n$  with $m_n < 1/N$.
Cover $K $ with those sets from ${\cal E}_n$ which contain at least one point of $K$.  This cover $U_N$ is a set in $B$.   
Since $K$ is compact, the distance from a point $x$ to $K$ is positive for each point $x$ in the complement of $K $. 
Thus $K = \cap_{N=1}^{\infty} U_N$, and the $\sigma$-algebra $\cal B$
contains all the closed subsets of $\Omega $, establishing (ii).   
\end{proof}

One would like to use ${\vecsp}_{\bf E}$ to construct a Hilbert space of functions on $\Omega $.
If $\Omega $ comes equipped with a measure $\mu $, the space $L^2(\mu )$ is an obvious candidate. 
The next result shows that a suitable measure $\rho $ may be constructed from a nested sequence of partitions
such as the canonical sequence from \propref{specialpart}.  

Letting ${\cal E}_0 = \Omega$, define $\rho (\emptyset ) = 0$ and $\rho ({\cal E}_0 ) = 1$.
Proceeding inductively, suppose $E \in {\cal E}_n$, and $\{ E(m), m=1,\dots , M \}$ is the partition of $E$ by sets $E(m) \in {\cal E}_{n+1}$.
Equally distribute the measure of $E$ to the sets $E_m$ by defining $\rho (E(m)) = \rho (E)/M$.
If $S$ is in the algebra $B$ of finite unions of sets in the partitions ${\cal E}_n$, then as above $S$ 
is the union of finitely many sets $S_k$ from a single partition ${\cal E}_N$.  
The definition $\rho (S) = \sum_k \rho (S_k)$ will not depend on the selected partition ${\cal E}_N$.

\begin{thm} \label{rhoex}
Assume that $\Omega $ is a nonempty, compact and totally disconnected metric space, 
with a nested sequence ${\bf E} = \{ {\cal E}_n \}$ of finite partitions by clopen sets.  Suppose $\lim _{n \to \infty} m_n = 0$.
The function $\rho :B \to \real $ extends to a Radon measure $\rho $ on $\Omega$, with
$\rho (I) > 0$ for every nonempty clopen set $I \subset \Omega $. 
The vector space ${\cal V}_{\bf E}$ is dense in $L^2(\rho)$. 
\end{thm}

\begin{proof}

The approach follows \cite{Folland}.

Suppose $\{ S_j \}$ is a sequence of disjoint sets from $B$ such that $S = \bigcup_{j=1}^{\infty} S_j \in B$.   
Every $S \in B$ is a clopen subset of $\Omega $.  Since the sets $S_j$ are disjoint and $S$ is compact, 
the cover by $\{ S_j \}$ is a finite cover, and $\rho (\bigcup_j S_j) = \sum_j \rho(S_j)$.  
Thus $\rho$ is a premeasure on $B$.  By \thmref{Borel} the sigma algebra ${\cal B}$ generated by $B$ is the Borel sets of $\Omega$.
Thus \cite[p. 30]{Folland} $\rho $ has a unique extension to a Borel measure (also called $\rho $) on $\Omega $.       

Suppose $I \subset \Omega $ is clopen, with $x \in I$.  Since $I$ is open,  there is a clopen set $U_x \in B$ 
containing $x$ and contained in $I$, so that $0 < \rho (U_x) \le \rho (I)$.  

By \cite[p. 210]{Folland} the measure $\rho $ is a Radon measure, and the continuous functions on $\Omega$ are dense in $L^2(\rho)$. 
By \thmref{Borel} any continuous function on $\Omega $ may be uniformly approximated by functions in ${\cal V}_{\bf E}$,  
so ${\cal V}_{\bf E}$ is dense in $L^2(\rho)$.

\end{proof}

Suppose $\Omega $ and ${\bf E}$ are as in \thmref{rhoex}.
Consider any positive Borel measure $\mu$ on $\Omega $, with $\mu (\Omega ) < \infty$ and $\mu (E) > 0$ for each $E \in {\cal E}_n$. 
The $L^2(\mu ) $ inner product $\langle f,g \rangle _{\mu}  = \int fg \ d \mu $
can be used to construct a generalized Haar basis on ${\cal V}_{\bf E} = \bigcup_{n=0}^{\infty} {\cal V}_n$.
A more general treatment, examining function theory in some depth, can be found in \cite{Girardi}.

The construction of the orthonormal basis $ \{ \chi _{m,n} \} $ for ${\cal V}_{\bf E}$ is inductive,
starting with $\chi _{0,0} = 1_{\Omega }\mu (\Omega )^{-1/2}$.
Suppose $E \in {\cal E}_n$, and $\{ E(m), m=1,\dots , M \}$ is the partition of $E$ by sets $E(m) \in {\cal E}_{n+1}$.
The $L^2(\mu )$ inner product on the span of the functions $1_{E(m)}$ gives
\[\langle 1_{E(j)} ,1_{E(k)} \rangle = \Bigl \{ \begin{matrix} \mu (E(k)), & j = k, \cr 0,& j \not= k \end{matrix} \Bigr \} .\]
Apply the Gram-Schmidt process to $1_{E}, 1_{E(1)},\dots ,1_{E(K-1)}$ 
to obtain $K$ orthonormal functions, with the first function in ${\cal V}_n$ and the remaining $K-1$ functions
in ${\cal V}_{n+1}$ and orthogonal to ${\cal V}_n$.

Let $\chi _{m,n}$ denote the constructed orthonormal basis of ${\cal V}_{\bf E}$.
The linear span of $\chi _{m,k}$ for $k \le n$ includes the characteristic function of each $E \in {\cal E}_{n}$,
so the span is ${\cal V}_{\bf E}$.  In addition $\lim _{n \to \infty} m_n = 0$, so by \propref{Borel} the linear span of the $\chi _{m,n}$ 
is uniformly dense in the continuous functions, giving the next result.

\begin{prop} \label{Hilb}
Assume that $\Omega $ is a nonempty, compact and totally disconnected metric space, 
with a nested sequence ${\bf E} = \{ {\cal E}_n \}$ of finite partitions by clopen sets.  Suppose $\lim _{n \to \infty} m_n = 0$.
Assume $\mu $ is a positive finite Borel measure on $\Omega $ with $\mu (E) > 0$ for each $E \in {\cal E}_n$.
Then the basis $ \{ \chi _{m,n} \} $ is a complete orthonormal basis for $L^2(\mu)$.
\end{prop}

\section{Dirichlet to Neumann map}

Given a Hilbert space basis such as $\{ \chi _{m,n} \}$, it is natural to look for distinguished operators $\dop$ on $L^2(\mu)$.
The values $\jmp$ where the number of $\epsilon$-components changes provides a decreasing positive sequence with limit zero
and finite multiplicity.  By a rough analogy with differential operators, once might consider $\lambda _0 = 0, \lambda _k = 1/\jmp$ as eigenvalues for $\dop$,   
with eigenfunctions $\{ \chi _{m,n} \}$.  This obviously encodes geometric data in the eigenvalue sequence.  
To reflect more of the structure of $\gbar$, the Dirichlet to Neumann map for harmonic functions on $\gbar$ will be used below to identify 
additional boundary measures, which can play a role similar to that of $\rho $ in \thmref{rhoex} and \propref{Hilb}.

A continuous function $f:\gbar \to \real $ is harmonic on $\graph _{int}$
if (i) $f'' = 0$ on each edge $e \in \edgeset $, and (ii) at each vertex $v \in \graph _{int}$ the derivative condition
\begin{equation} \label{Kirchhoff}
\sum_{e \sim v} \partial _{\nu} f_e(v) = 0,
\end{equation}
holds, where the sum is taken over edges incident on $v$, denoted $e \sim v$.  
Recall that each edge $e$ is identified with an interval $[a_e,b_e]$.  Here, $\partial _{\nu} f_e$ denotes the derivative of $f$ along the edge $e$ 
when $v$ is identified with $b_e$. 
(This choice is made for consistency with \cite{Carlson12}. The opposite sign is chosen in \cite[p. 88]{BK}.)

Since $f$ is continuous at $v$ and linear on edges $e$ incident on $v$, 
\[f_e(a_e) = f(v) - \partial _{\nu}f_e(v)(b_e-a_e),\] 
or \[\frac{f_e}{l_e} = \frac{f(v)}{l_e} - \partial _{\nu} f_e(v).\]
Using $C_e= (b_e - a_e)^{-1} = l_e^{-1}$ and adding gives
\begin{equation} \label{diffharm}
f(v) = (\sum _{e \sim v} C_e)^{-1} \sum _{e \sim v} C_e f_e(a_e) .
\end{equation}
Thus the definition of harmonic functions here is consistent with the definition commonly used for resistor networks and reversible Markov chains
\cite[p. 46]{Doyle}.   Since $f(v)$ is a weighted average of the neighboring values $f(a_e)$ for each $v \in \graph _{int}$,
$f$ satisfies the maximum principle: 
that is, the minimum or maximum of $f$ is achieved on $\partial \gbar$, and if $\graph _{int}$ is connected then
$f$ can only have a minimum or maximum at a point $x \in \graph _{int}$ when $f$ is constant.

The subsequent developments build on \cite{Carlson12}, where somewhat more general problems were treated.
In the context of this work, \cite[Thm 2.7]{Carlson12} has the following conclusion, showing that the Dirichlet problem is solvable.
In that work $\partial \gbar \cap \graph $ was the set of vertices of degree one.

\begin{thm} \label{DX}
If $\gbar$ is weakly connected and compact, then every continuous function $F:\partial \gbar \to \real $ 
has a unique extension to a continuous function $f: \gbar \to \real $ that is harmonic on $\graph _{int}$.
\end{thm}

In \cite{Carlson08,Carlson12} a central role is played by an algebra $\alg$ of 'eventually flat' functions.
These are functions $\phi: \gbar \to \real $ which are continuous, infinitely differentiable on the open edges of $\graph $,  
with $\phi ' = 0$ in the complement of a finite collection of edges, and 
in an open neighborhood of each vertex $v \in \graph$.  The following fact about $\alg$ is from \cite[ Lem. 3.5]{Carlson08}.

\begin{prop} \label{Aprop}
Suppose $\gbar$ is weakly connected.  If $\Omega $ and $\Omega _1$ are disjoint compact subsets of $\gbar $, then there is a 
function $\phi \in \alg $ such that $0 \le \phi \le 1$,
\[\phi (x) = 1, x \in \Omega ,\quad \phi (x) = 0, x \in \Omega _1.\]
\end{prop}   
 
A class of functions similar to $\alg$, and its relation to the end compactification of a graph, appeared previously in \cite{Cartwright},  
 
\subsection{Dirichlet to Neumann map for finite graphs}

Suppose $\graph _0$ is a finite metric graph without isolated vertices and at least two edges.
If $\graph _0$ is a subgraph of $\graph$, one may wish to include in $\partial \graph _0$ those vertices
with incident edges from both $\graph _0$ and its complement.  
With this in mind, let $\partial \graph _0$ to be any fixed set of vertices from $\graph _0$ which
includes the vertices of degree $1$.

Despite the more general definition of boundary vertices for $\graph _0$, \thmref{DX} can be used to establish
the existence and uniqueness of harmonic functions by a reduction to the previous case.  
Suppose $v \in \partial \graph _0$ with $\deg (v) = K$.
Decouple the edges $e_k$ incident on $v$ by replacing $v$ with new and distinct vertices $v_k$,
each now having degree $1$.  Doing this for each $v \in \partial \graph _0$ replaces $\graph _0$ with  
a new graph $\graph _1$ whose boundary vertices all have degree $1$.  If $U$ is a function on $\partial \graph _1$
with $U(v_1) = \dots = U(v_K)$ for each $v \in \partial \graph _0$, then $U$ has a harmonic extension $u$ to $\graph _1$.
By using the natural identifications, $u$ will also define a unique harmonic function on $\graph _0$ with the given values at  $v \in \partial \graph _0$. 

Simple calculus will lead to the definition and basic properties of the Dirichlet to Neumann map on finite graphs.
If $f$ and $g$ have continuous second derivatives on each edge, are continuous at the vertices, and satisfy \eqref{Kirchhoff} at interior vertices, then
integration by parts, together with the vertex conditions, gives
\begin{equation} \label{IBP1}
 \int_{\graph _0} -f'' g = - \sum_{v \in \partial \graph _0} \sum_{e \sim v} (\partial _{\nu} f_e(v)) g(v) + \int_{\graph _0} f'g' .
\end{equation}
Reversing the roles of $f$ and $g$ leads to 
\begin{equation} \label{IBP2}
\int_{\graph _0} -f'' g = \sum_{v \in \partial \graph _0} \sum_{e \sim v} [f(v)\partial _{\nu} g _{e}(v) - g(v)\partial _{\nu} f_e(v)] - \int_{\graph _0} f g''.
\end{equation}
If $f'' = 0$, \eqref{IBP1} becomes
\begin{equation} \label{Finform}
 \sum_{v \in \partial \graph _0} \sum_{e \sim v} (\partial _{\nu} f_e(v)) g(v) = \int_{\graph _0} f'g' ,
\end{equation}
while if $f'' = g'' = 0$, then \eqref{IBP2} reduces to 
\begin{equation} \label{symm}
\sum_{v \in \partial \graph _0} g(v) \sum_{e \sim v}  \partial _{\nu } f_e(v) = \sum_{v \in \partial \graph _0} f(v) \sum_{e \sim v}  \partial _{\nu }g_e(v) . 
\end{equation}
For $F:\partial \graph _0 \to \real $ with harmonic extension $f$, the basic Dirichlet to Neumann map $\Lambda $ is defined by 
\[ \Lambda _{\mu} F(v) = \sum_{e \sim v} \partial _{\nu } f_e(v) , \quad v \in \partial \graph _0.\]

Suppose now that $\mu$ is a measure on $\partial \graph _0$ with $\mu (v) > 0$ for all $v \in \partial \graph _0$.
The vector space $\vecsp$ of real valued functions on $\partial \graph _0$ has $\mu$ dependent
inner products and Dirichlet to Neumann maps.
If $\partial \graph _0 = \{ v_1,\dots ,v_N \}$ then an inner product  on $\partial \graph _0$ is
\[ \langle F,G \rangle _{\mu}  = \frac{1}{\mu ( \partial \graph _0) }\sum_{v \in \partial \graph _0} F(v) G(v) \mu (v).\]
If $f$ is the harmonic extension of $F$ to $\graph _0$, the $\mu $ dependent Dirichlet to Neumann map is 
\begin{equation} \label{genDN}
\Lambda _{\mu} F(v) = \mu (v)^{-1} \sum_{e \sim v} \partial _{\nu } f_e(v) , \quad v \in \partial \graph _0.
\end{equation}
With these definitions the Dirichlet to Neumann maps $\Lambda _{\mu}$ behave in the expected way.

\begin{prop} \label{DNfinite}
For a finite graph $\graph _0$ with boundary measure $\mu$ the Dirichlet to Neumann map $\Lambda _{\mu} $ is self-adjoint and nonnegative. 
\end{prop}

\begin{proof}
The desired symmetry follows from \eqref{symm},
\[ \langle \Lambda _{\mu}F,G \rangle _{\mu} 
=  \frac{1}{\mu ( \partial \graph _0) }\sum_{v \in \partial \graph _0} [\mu (v)^{-1} \sum_{e \sim v} \partial _{\nu } f_e(v)] g(v) \mu (v) \]
\[ =  \frac{1}{\mu ( \partial \graph _0) }\sum_{v \in \partial \graph _0} g(v) [ \sum_{e \sim v} \partial _{\nu } f_e(v)]   
=  \frac{1}{\mu ( \partial \graph _0) }\sum_{v \in \partial \graph _0} f(v) [ \sum_{e \sim v} \partial _{\nu } g_e(v)] \]
\[ =  \frac{1}{\mu ( \partial \graph _0) }\sum_{v \in \partial \graph _0} f(v) [\mu (v)^{-1} \sum_{e \sim v} \partial _{\nu } g_e(v)]\mu (v) =  \langle F, \Lambda _{\mu}G \rangle _{\mu}.  \]

The lower bound for $\Lambda _{\mu} $ comes from \eqref{Finform} if $f = g$ and $f'' = 0$.
\end{proof}

Compressions of the Dirichlet to Neumann map can be associated to partitions of the graph boundary.
These compressions will help with a comparison of the Dirichlet to Neumann maps on the boundary of a graph completion $\partial \gbar $
with the boundary of certain finite subgraphs.
Suppose then that ${\cal E} = \{ E(n), n=1,\dots ,N \}$ is a partition of the boundary $\partial \graph _0$ of a finite graph.  
A subspace ${\vecsp}_{\cal E}$ of $\vecsp$ is given by the span of the characteristic functions of the sets $E(n)$.
Functions $F \in {\vecsp}_{\cal E}$ may be viewed as having domain ${\cal E}$, with values $F(E(n))$.
This subspace inherits the inner product
\begin{equation} \label{partip}
\langle F,G \rangle _{\cal E, \mu} = \frac{1}{\mu (\partial \graph _0)} \sum _n F(E(n))G(E(n)) \mu (E(n)).
\end{equation}

Let $P_{\cal E}$ denote the orthogonal projection from $\vecsp$ to ${\vecsp}_{\cal E}$.
A compressed Dirichlet to Neumann map $\Lambda _{\cal E,\mu }:{\vecsp}_{\cal E} \to {\vecsp}_{\cal E}$
is then given by $P_{\cal E}\Lambda _{\mu}P_{\cal E}$.  Denote by $\Lambda _{\cal E, \mu}: {\vecsp}_{\cal E} \to {\vecsp}_{\cal E}$
the restriction of $P_{\cal E}\Lambda _{\mu}P_{\cal E}$ to ${\vecsp}_{\cal E}$.
An orthonormal basis for ${\vecsp}_{\cal E}$ is $\{ 1_{E(n)}\sqrt{\mu (\partial \graph _0)/ \mu (E(n))} \}$, so
if $F \in {\vecsp}_{\cal E}$ has harmonic extension $f$, the compressed map is
\begin{equation} \label{fcompdef}
(\Lambda _{\cal E, \mu}F)(E(n))  = \mu (E(n))^{-1} \sum_{v \in E(n)} \sum_{e \sim v} \partial _{\nu}f_e(v). 
\end{equation}  
As a compression of $\Lambda _{\mu}$, the map $\Lambda _{\cal E, \mu}: {\vecsp}_{\cal E} \to {\vecsp}_{\cal E}$ is
self-adjoint and nonnegative.

\subsection{Extension to infinite graphs} 

An alternate description of the Dirichlet to Neumann map $\Lambda _{\mu }$ facilitates an extension to infinite graphs.
Suppose $f$ is continuous on finite graph $\graph _0$ and harmonic in the interior.
Let $E$ be a subset of $\partial \graph _0$, with $E^c = \partial \graph _0 \setminus E$.
Pick a smooth function $\phi :\graph _0 \to \real $ whose derivative, along each edge, vanishes at each vertex, with $\phi (v) = 1$ for $v \in E$,
and $\phi (v) = 0$ for $v \in E^c$. Then \eqref{IBP2} gives
\begin{equation} \label{Fweakform}
- \int_{\graph _0} f \phi '' = \sum_{v \in E} \sum_{e \sim v} \partial _{\nu} f_e(v) .
\end{equation}
The left side of \eqref{Fweakform} is available if $\gbar $ is a compact, weakly connected metric graph completion
and $E$ is a nonempty clopen subset of $\partial \gbar $. 
By \propref{Aprop} a function $\phi $ can be found in the algebra $\alg$ of eventually flat functions, with $\phi (x) = 1$ for $x \in E$ and
$\phi (x) = 0$ for $x \in E^c$.  
For a continuous function $F:\partial \gbar \to \real $ with harmonic extension $f$, and 
a measure $\mu $ satisfying $\mu (E) > 0$, define the Dirichlet to Neumann map of $F$ on the set $E$ as
\begin{equation} \label{DNdef}
\Lambda _{\mu} F (E) = - \mu (E)^{-1}  \int_{\graph } f \phi '' .
\end{equation}

\begin{lem} \label{nupos}
Assume that $E$ is a nonempty clopen subset of $\partial \gbar$, and $E^c = \partial \gbar \setminus E$.
Suppose $F: \partial \gbar \to \real$ is continuous,
\[ F(x) =  0,  \ x \in E, \quad F(x) > 0, \  x \in E^c,\]
and $f$ is the harmonic extension of $F$ to $\graph $.
If $\phi \in \alg$ satisfies $\phi (x) = 1$ for $x \in E$ and $\phi (x) = 0$ for $x \in E^c$, and   then 
\begin{equation} \label{basecase}
 \int_{\graph } f \phi '' > 0.
\end{equation}
\end{lem}

\begin{proof}

The plan is to replace $\gbar$ with an approximating finite graph, where \eqref{Fweakform} can be used.
Note that $f(x) > 0$ for $x \in \graph $ since $f$ is a nonconstant harmonic function with nonnegative boundary values. 
With no loss of generality, assume that $\max_{x \in \gbar} f(x) = 1$.
Begin by choosing $\epsilon  > 0$ such that $\phi '(x) = 0$ when $d(x,\partial \gbar ) < \epsilon  $.

Since $\gbar $ is connected, the range of $f$ is $[0,1]$.  An argument by contradiction will show that if $t \in (0,1)$ is close enough to $0$,
then $d(x,E) < \epsilon $ for every $x$ in the level set $f^{-1}(t)$.  
So assume there is a sequence $t_n \to 0$ and points $x_n \in f^{-1}(t_n)$ with $d(x_n, E ) \ge \epsilon $. 
By compactness of $\gbar$ the sequence $\{x_n \}$ has a subsequential limit $z$, with $f(z) = 0$ by continuity of $f$.
Since $z \notin E^c$ and $d(z,E ) \ge \epsilon $, $z$ must be in $\graph $, which is impossible.  

Pick $t_0 \in (0,1)$ so that for all $t$ with $0 < t < t_0 $ the level sets  $f^{-1}(t)$ lie in the $\epsilon $ neighborhood of $E$. 
Since the vertex set of $\graph $ is countable, as is the set of edges $e$ where $f(x)$ is constant, it is possible to choose
$t$ so $f^{-1}(t)$ contains no vertices, and (since $f$ is linear on each edge) no points $x$ with $f'(x) = 0$.

Next define a graph $\graph _1$ whose vertices are (i) vertices of $\graph $ in $f^{-1}((t,1])$ and (ii) those points $x \in \graph $ with $f(x) = t$. 
Edges of $\graph _1$ include edges joining two vertices in $f^{-1}((t,1])$.  In addition,
if an edge $e$ of $\graph $ has vertices $v_1,v_2$ with $f(v_1) < t$ and $f(v_2) > t$, then 
replace $e$ with a new edge $e_1$ having vertices $x,v_2$ and add $e_1$ to $\graph _1$.

Finally, trim $\graph _1$ by retaining those edges, with their vertices, containing a point whose distance from $E^c$ is at least $\epsilon /2$
The resulting graph $\graph _2$ includes all points in $\graph $ where $\phi '' \not= 0$, so 
\[ \int_{\graph } f \phi '' = \int_{\graph _2} f \phi ''.\]
If necessary $t$ may be reduced so that $t$ is the minimum of $f$ on $\graph _2$.
Since $v \in f^{-1}(t)$ are the minima of $f$ on $\graph _2$, $\partial _{\nu} f(v) < 0$.
Finally, \eqref{Fweakform} gives
\[ \int_{\graph _2} f \phi '' = -\sum_{v \in f^{-1}(t)} \sum_{e \sim v} \partial _{\nu} f_e(v) > 0 .\]
\end{proof}

\thmref{rhoex} established the existence of certain measures $\rho$ generated by nested sequences ${\bf E} = \{ {\cal E}_n \}$ 
of finite clopen partitions of compact, nowhere dense sets $\Omega $, in particular for $\partial \gbar $.
Additional measures for $\partial \gbar $ may be defined using \eqref{basecase} on an extended boundary for $\graph $.

Pick a vertex $w \in \graph $ which will be treated as an added and distinguished boundary vertex.
To maintain focus on the original boundary, the new boundary will be denoted $\partial \gbar \cup \{ w \}$.
The function $1_{\{ w \}}$ defined on $\partial \gbar \cup \{ w \}$ will have a harmonic extension $f$
on the rest of $\gbar$.  For a clopen set $E \in \partial \gbar$, take $\phi \in \alg $ with $\phi (x) = 1$ in an open neighborhood of $E$ and $\phi (x) = 0$ 
in an open neighborhood of $E^c$ in $\{ w \} \cup \partial \gbar $.
Define 
\begin{equation} \label{nudef}
\nu _w (E) =  \int_{\graph } f \phi ''.
\end{equation}

\begin{thm} \label{nuex}
The function $\nu _w $ extends to a nonnegative Radon measure on $\partial \gbar$, with
$\nu _w (S) > 0$ for every nonempty clopen set $S \subset \partial \gbar $. 
\end{thm}  

\begin{proof} 
With the adjusted boundary for $\gbar$, the positivity $\nu _w (E) > 0$ was established in \lemref{nupos}.

Let $B$ denote the algebra of clopen subsets of $\partial \gbar $.   
Suppose $E$ is the union of a finite collection of disjoint clopen sets, $E = \bigcup_{n=1}^N E(n) $.
There are functions $\phi _n \in \alg$ with $\phi _n(x) = 1$ in an open neighborhood of $E(n)$ and 
$\phi _n(x) = 0$ in an open neighborhood of $E^c$ in $\{ w \} \cup \partial \gbar $. 
Then \eqref{basecase} shows that 
\[\nu _w(E) =  \int_{\graph } f \phi '' =  \int_{\graph } f \sum_{n=1}^N \phi _n '' = \sum_{n=1}^N \nu_w(E(n)), \]
and $\nu _w$ is additive on $B$.   The $\sigma $-algebra $\cal B$ generated by $B$ is the Borel subsets of $\partial \gbar $.
As in \thmref{rhoex}, $\nu _x$ has a unique extension to a nonnegative Radon measure on $\cal B$.

\end{proof}

\begin{cor}
Suppose $w(1),w(2) \in \graph $.  Then there is a $C > 0$ such that $C \nu _{w(1)} \ge \nu _{w(2)}$ on $\partial \gbar$,
so that  $\nu _{w(1)}$ and $ \nu _{w(2)}$ are mutually absolutely continuous.
\end{cor}

\begin{proof}
Pick points $w(1),w(2) \in \graph $.  
For $j = 1,2$ let $f_j$ be the harmonic extension of $1_{w(j)}:w(j) \cup \partial \gbar \to \real $ to $\gbar $.  
Since $0 < f_1(v) < 1$ for $v \in \graph $, there is a positive $C$ such that $Cf_1(w(2)) - f_2(w(2))  > 0$, while $Cf_1 - f_2  = 0$ on $\partial \gbar $.  
By \lemref{nupos}  $ C \nu _{w(1)} - \nu _{w(2)} > 0$ on each clopen subset of $\partial \gbar $.

A Radon measure $\mu $ has the following regularity properties \cite[p. 205]{Folland}:
(i) for all Borel sets $S$,
\[ \mu (S) = \inf \{ \mu (U), U \ {\rm open} \ {\rm and} \ E \subset U \},\]
and (ii) for all open $U$,
\[\mu (U) = \sup \{ \mu (K), K \ {\rm compact} \ {\rm and} \ K \subset U \}.\]
Suppose $K \subset \partial \gbar $ is compact and $\mu = C \nu _{w(1)} - \nu _{w(2)}$.
By \propref{manycor}, for any open set $U$ containing $K$ there is a clopen set
$E \subset U$ containing $K$.  By (i), $\mu (K) \ge 0$, by (ii) any open $U$ satisfies $\mu (U) \ge 0$, and by
(i) again $\mu (S) \ge 0$ for all Borel sets in $\partial \gbar $. 

\end{proof}

If ${\cal E} = \{ E(n), n = 1, \dots ,N \}$ is a finite partition of $\partial \gbar $ by clopen sets, the definition of  
the compressed Dirichlet to Neumann map may be extended to $\vecsp _{\cal E}$.   The inner product \eqref{partip}
and orthonormal basis $\{ 1_{E(n)}\sqrt{\mu (\partial \gbar)/ \mu (E(n))} \}$ are as before.  If $F \in \vecsp _{\cal E} $ has harmonic extension $f$,
then
\begin{equation} \label{bndmap}
 \Lambda _{{\cal E}, \mu }F(E(n)) = - \mu (E(n))^{-1} \int_{\graph } f\phi _n'' .
 \end{equation}

The next result will show that $\Lambda _{{\cal E},\mu} $ is the limit of 
compressed Dirichlet to Neumann maps on the finite subgraphs $\graph _{\epsilon}$ discussed in \propref{partition}.    
Using membership of boundary vertices of $\graph _{\epsilon}$ in the neighborhoods
$N_{\epsilon}(E(n))$, the partition $\cal E $ induces a corresponding partition ${\cal E}(\epsilon) = \{ \wt{E(n)}, n = 1, \dots ,N \}$ 
of $\graph _{\epsilon}$ for $\epsilon > 0$ and sufficiently small.  The partitions ${\cal E}(\epsilon)$ will inherit  
a measure $\mu $ from $\cal E$. The identification $\wt{E(n)}$ with $E(n)$ extends to the vector space ${\vecsp}_{\cal E}$.
For sufficiently small $\epsilon > 0$, $\Lambda _{{\cal E}(\epsilon), \mu }:{\vecsp}_{\cal E} \to {\vecsp}_{\cal E}$ 
will denote the compressed Dirichlet to Neumann map on $\partial \graph _{\epsilon}$.

\begin{thm} \label{Biglim}
\begin{equation} \label{Biglimeq}
\lim_{\epsilon \to 0^+} \Lambda _{{\cal E}(\epsilon), \mu } = \Lambda _{\cal E ,\mu}.
\end{equation}
Consequently, $\Lambda _{\cal E, \mu}:{\vecsp}_{\cal E} \to {\vecsp}_{\cal E}$ is self adjoint and nonnegative.
\end{thm}

\begin{proof} Since ${\vecsp}_{\cal E}$ is finite dimensional, it suffices to check that 
\[\lim_{\epsilon \to 0^+} \Lambda _{{\cal E}(\epsilon), \mu }F = \Lambda _{\cal E,\mu}F\]
for basis vectors $F = 1_{E(n)}$ of $ {\vecsp}_{\cal E}$.
Let $f$ denote the harmonic extension of $F$ from $\partial \gbar$ to $\gbar$, while 
$f_1$ denotes the harmonic extension of $F$ from $\partial \graph _{\epsilon}$ to $\graph _{\epsilon}$.
By the uniform continuity of $f$ on the compact set $\gbar $, for every $\sigma > 0$
there is a $\delta $ such that $|f(x) - f(y)| < \sigma $ when $d(x,y) < \delta $.   

By \propref{Aprop} there are functions $\phi \in \alg$ with the value $1$ on $E(m)$ and $0$ on the complement in $\partial \gbar$.
Since $\phi ' = 0$ except on a finite subgraph of $\graph $, there is an $\epsilon _0 > 0$ such that, if $\epsilon < \epsilon _0 $, 
then $\phi '' = 0$ on $ \graph _{\epsilon} \setminus \graph _{\epsilon _0}$. 

Suppose $ 0< \epsilon < \epsilon _0$ and $\epsilon < \delta $.  For $y \in \partial \graph _{\epsilon}$ we have $|F(y) - f(y)| < \sigma $. 
By the maximum principle, $|f_1(x) - f(x)| < \sigma $ for all $x \in \graph _{\epsilon }$. Since $\phi ''(x) = 0$ for $x \in \gbar \setminus \graph _{\epsilon}$,
\[|\Lambda _{\cal E,\mu }F(E(m)) - \Lambda _{{\cal E}(\epsilon),\mu }F(E(m))| = |\mu (E(m))^{-1} \int_{\graph _{\epsilon}} (f - f_1) \phi ''| \]
\[\le \sigma \mu (E)^{-1} \int_{\graph _{\epsilon}} |\phi ''|. \]
The last integral does not change for $\epsilon < \epsilon _0$, and $\sigma \to 0$ as $\epsilon \to 0^+$.
In particular this shows that the matrices for the finite dimensional linear maps $ \Lambda _{{\cal E}(\epsilon),\mu } $
with respect to the given orthonormal basis converge to the matrix for $\Lambda _{\cal E,\mu }$
as $\epsilon \to 0^+$.  Since the matrices for $\Lambda _{\cal E(\epsilon),\mu }$ are self adjoint and nonnegative, the same holds 
for $\Lambda _{\cal E,\mu }$.

\end{proof}

 \section{Appendix: A Dirichlet problem counterexample}
 
Despite considerable work developing the theory of harmonic functions on infinite graphs \cite{Cartwright} \cite[p. 474-486]{LP},   
the author struggled to find simple examples in the literature showing that in some cases the Dirichlet problem is not solvable.   
This appendix provides such an example, showing that the Dirichlet problem may not be solvable if the compactness condition on $\gbar$ is relaxed.  

Start with a sequence of vertices $v_1,v_2,v_3, \dots $ with $v_n$ adjacent to $v_{n+1}$.
Assume that $d(v_n,v_{n+1}) = 1/n^2$.
For each $n >1$, add $M_n$ additional vertices $w_{n,m}$.  Values for $M_n$ will be chosen later.  
Each $w_{n,m}$ is adjacent to the vertex $v_n$, with $d(v_n,w_{n,m}) = 1$.  $\graph$ has no other vertices or edges.    
The completion $\gbar$ has one limit point $\omega = \lim_n v_n$.  
Note that $\gbar$ is weakly connected with finite diameter.

The boundary of $\gbar $ consists of $\{ \omega , v_1, w_{n,m} \}$.
As elements of $\partial \gbar$ the boundary points are all isolated, so 
a continuous function on the boundary is given by assigning any real values to these points.  

A simple argument shows that for suitably chosen $M_n$ there is no harmonic function $f$ with $f(\omega ) \not= 0$, but with $f$ vanishing 
at the other boundary points.  First consider the case when $f$, which must be linear on each edge, is $0$ on the edge $(v_1,v_2)$.
Since $f$ is zero at $v_2$ and at each vertex $w_{2,m}$, $f$ is also zero on each edge $(v_2,w_{2,m})$.  
The derivative condition at $v_2$ forces $f$ to vanish on $(v_2,v_3)$, by induction $f$ is the zero function on $\graph $,
and $f(\omega ) = \lim_{n \to \infty} f(v_n) = 0$.

If $f$ is not $0$ on $(v_1,v_2)$, then after rescaling we may assume $f(v_2) = 1$.
Looking outward from $v_2$, the derivatives have the value $-1$ on the edges 
$(v_1,v_2)$ and $(v_2,w_{2,m})$.  The derivative on $(v_2,v_3)$ must be $1+ M_2$,
and so $f(v_3) = f(v_2) + (1+M_2)/2^2$.  For subsequent vertices 
$f(v_{n+1}) > f(v_n) + M_n/n^2$, so for $M_n$ sufficiently large, $f(v_n) \to \infty $.
Thus $f$ cannot be continuously extended to $\omega$.

\bibliographystyle{amsalpha}

\end{document}